\documentclass[a4paper,12pt,showkeys]{amsart}

\usepackage{amsmath,amssymb,amsthm,mathrsfs}
\usepackage{type1cm}
\usepackage{ulem}
\usepackage{fancybox}
\usepackage[dvipdfmx]{graphicx}
\usepackage{tikz}
\usepackage[all]{xy}
\usepackage{array,booktabs}
\usepackage{bm}

\theoremstyle{plain}
\newtheorem{theorem}{Theorem}[section]

\newtheorem{lemma}[theorem]{Lemma}

\theoremstyle{definition}
\newtheorem{definition}[theorem]{Definition}

\newtheorem{example}[theorem]{Example}
\newtheorem{proposition}[theorem]{Proposition}

\newcommand{\GL}{{\rm{GL}}}
\newcommand{\PGL}{{\rm{PGL}}}

\newcommand{\Ker}{{\rm{Ker}}}

\newcommand{\Hom}{{\rm{Hom}}}
\newcommand{\End}{{\rm{End}}}

\newcommand{\Lim}{\displaystyle \lim_{t\rightarrow 0}}
\newcommand{\PM}{\mathbb{P}M}
\newcommand{\Aut}{{\rm{Aut}}}
\newcommand{\Out}{{\rm{Out}}}

\begin{document}

\title{On $\mathbb{P}M$-mapping class monoids}
\author{Toshinori Miyatani}
\date{}
\address{Graduate School of Science, Hokkaido University, Sapporo, 004-0022, Japan}
\email{miyatani@math.sci.hokudai.ac.jp}

\begin{abstract}
In this paper, we introduce $\mathbb{P}M$-mapping class monoids. Braid groups and mapping class groups have many features in common. Similarly to the notion of braid $\mathbb{P}M$-monoid, $\mathbb{P}M$-mapping class monoid is defined. This construction is an analogy of inverse mapping class monoid defined by R. Karoui and V. V. Vershinin. As the main result, we give the analog of the Dehn-Nilsen-Baer theorem. 
\end{abstract}

\keywords{monoids; braid monoids; compactifications; mapping class groups: mapping class monoids.}

\maketitle
\tableofcontents

\section{Introduction}
Mapping class groups is an important object in many areas in mathematics. Analogically, we would like to define the notion of mapping class monoids. In \cite{M}, we introduced the $\mathbb{P}M$-monoids and braid $\mathbb{P}M$-monoids in the context of a compactification of the projective linear group defined by Mutsumi Saito \cite{S1}. We will introduce the notion of $\mathbb{P}M$-mapping class monoids and observe the relation to the braid $\mathbb{P}M$-monoid. The construction of $\mathbb{P}M$-mapping class monoid is an analogy of inverse mapping class monoid defined by R. Karoui and V. V. Vershinin \cite{V2}. There exists an important theorem in the theory of mapping class group called Dehn-Nilsen-Baer theorem. This is stated as follows: groups of isotopy class of homeomorphisms are described in terms of an automorphism of the fundamental group of the corresponding surface. As a main result, we will show Dehn-Nilsen-Baer theorem for $\mathbb{P}M$-mapping class monoids.

This paper is organized as follows. In Section 2.1, we explain the compactification of the projective linear group. In Section 2.2 and 2.3, we review mapping class groups and inverse mapping class monoids. In Section 4, we define $\mathbb{P}M$-mapping class monoids and calculate some examples. In Section 5.1 and 5.2 we review Dehn-Nilsen-Baer theorem for mapping class groups and inverse mapping class monoids. In Section 5.3, we prove the Dehn-Nilsen-Baer theorem for $\mathbb{P}M$-mapping class monoids.

\section{Compactification of $\PGL$ and $\mathbb{P}M$-monoids}

\subsection{Compactification of $\PGL$}
We explain the compactification of the projective linear group constructed by M. Saito \cite{S1}. 
\subsubsection{Motivation}
One strategy of compactification is constructing a ``limit''. Then we consider the set of all limit points and introduce a topology compatible with the limit. For instance Y. A. Neretin constructed a compactification of the projective linear group by this strategy called hinge \cite{N1}. 

Let $V$ be an $n$-dimensional vector space over $\mathbb{C}$ and $A_i\in \End(V)$, $(i=1,2,\dots)$. Suppose that the linear map
\begin{equation*} \label{Ae}
A_{\epsilon}:=\displaystyle \sum_{i=0}^{m} A_i \epsilon^i
\end{equation*}
is in $\GL(V)$ for $\epsilon\neq 0$. Dividing by nonzero scalar matrices we consider the projective linear map
\begin{equation} \label{Ae-}
\overline{A_{\epsilon}}\in \PGL(V).
\end{equation}
We want to define a ``limit'' $\displaystyle \lim_{\epsilon\rightarrow 0}\overline{A_{\epsilon}}$. To define a limit, we observe the action of $\overline{A_{\epsilon}}$ on $\mathbb{P}(V)$. For $\overline{x}\in\mathbb{P}(V)$ we have 
\begin{equation*}
\displaystyle \lim_{\epsilon\rightarrow 0}\overline{A_{\epsilon}(x)}=
\begin{cases}
\overline{A_0x} & (x\not\in \Ker A_0) \\
\overline{A_1x} & (x\not\in \Ker A_0\backslash \Ker A_1) \\
\overline{A_2x} & (x\not\in \Ker A_0\cap \Ker A_1\backslash \Ker A_2) \\
\,\,\,\,\vdots  & \,\,\,\,\,\,\,\,\,\,\,\,\,\,\,\,\,\,\,\,\,\,\,\,\,\,\,\,\,\,\,\,\,\,\,\,\,\,\,\,\,\,\,\,\,\,\,\,\,\,\,\,\,\,\,\,\,\,\,\,\,\,\,\,\,\,\,\,\,\,\,\,\,\,\,.
\end{cases}
\end{equation*}
Thus we define the limit of (\ref{Ae-}) as
\begin{equation} \label{li}
\displaystyle \lim_{\epsilon\rightarrow 0}\overline{A_{\epsilon}}:=(\overline{A_0},\overline{A_1}|_{\mathbb{P}(\Ker A_0)},\overline{A_2}|_{\mathbb{P}(\Ker A_0\cap \Ker A_1)},\dots).
\end{equation}
\subsubsection{Definition of $\mathbb{P}M$}
In order to construct a compactification of the projective linear group, we consider the set of forms of the right hand side of (\ref{li}). We define the following sets.
Let $V$ be an $n$- dimensional vector space over $\mathbb{C}$. Set
\begin{equation*}
M:=M(V)=\left\{(A_0,A_1,\dots,A_m)\mid 
\begin{matrix} m=0,1,2,\dots \\ 0\neq A_i\in \Hom(V_i,V)\,\, (0\le i\le m) \\ V_0=V, V_{m+1}=0 \\ V_{i+1}=\Ker (A_i)\,\,(0\le i\le m)
\end{matrix}
\right\}
\end{equation*}
and
\begin{equation*}
\widetilde{M}:=\widetilde{M}(V)=\left\{(A_0,A_1,\dots,A_m)\mid 
\begin{matrix} m=0,1,2,\dots \\ 0\neq A_i\in \End(V)\,\, (0\le i\le m) \\ \displaystyle \cap_{k=0}^{i-1}\Ker A_k\not\subseteq \Ker A_i \\ \displaystyle \cap_{k=0}^{m}\Ker A_k=0
\end{matrix}
\right\}.
\end{equation*}
Let $\mathcal{A}:=(A_0,A_1,\dots,A_m)\in M$. Since $A_i\in \Hom(V_i,V)\backslash \{0\}$, we can consider the element $\overline{A_i}\in \mathbb{P}\Hom(V_i,V)$ represented by $A_i$, and we can define
\begin{equation*}
\mathbb{P}\mathcal{A}:=(\overline{A_0},\overline{A_1},\dots,\overline{A_m}).
\end{equation*}
Let $\mathbb{P}M=\mathbb{P}M(V)$ denote the image of $M$ under $\mathbb{P}$. $\mathbb{P}\widetilde{M}$ can be defined similarly.

\subsubsection{Topology of $\mathbb{P}M$}
We introduce a topology in $\PM$ which we can deal with the limit (\ref{li}). We fix a Hermitian inner product on $V$. Let $W$ be a subspace of $V$. By considering $V=W\oplus W^{\bot}$ via this inner product, we regard $\Hom(W,V)$ as a subspace of $\End(V)$. We consider the classical topology in $\mathbb{P}\Hom(W,V)$ for any subspace $W$ of $V$. \\
Let $\mathbb{A}=(A_0,A_1,\dots,A_m)\in M$. Then $A_i\in\Hom(V_i,V)$, where $V_i=V(\mathbb{A})_i=\Ker(A_{i-1})$. Let $U_i$ be a neighborhood of $\overline{A_i}$ in $\mathbb{P}\Hom(V(\mathbb{A})_i, V)$. Then set 
\begin{equation} \label{nbd}
\begin{split}
U_{\mathbb{PA}}(U_0,\dots,U_m)
=&\Biggl\{\mathbb{PB}=(\overline{B_0},\overline{B_1},\dots,\overline{B_n}) \\ & \mid  
\begin{matrix} ^{\forall}i=1,\dots,m, ^{\exists}j\in \{1,\dots,n\} \text{ s.t. } \\ V(\mathbb{B})_j\supseteq V(\mathbb{A})_i \text{ and } 
\\ \overline{B_j|_{V(\mathbb{A})_i}}\in U_i
\end{matrix}
\Biggr\}.
\end{split}
\end{equation}
We will exhibit now the sets (\ref{nbd}) define a topology that can deal with the limit (\ref{li}) by using the following example. 
\begin{example}
Let V be a 4-dimensional vector space over $\mathbb{C}$. Taking the standard basis, we identify $V\cong \mathbb{C}^4$. Let 
\begin{equation*}
\mathbb{A}=(A_0,A_1,A_2,A_3)=\left(
\begin{pmatrix} 
1&0&0&0 \\
0&0&0&0 \\
0&0&0&0 \\
0&0&0&0
\end{pmatrix},
\begin{pmatrix}
0&0&0 \\
1&0&0 \\
0&0&0 \\
0&0&0
\end{pmatrix},
\begin{pmatrix}
0&0 \\
0&0 \\
1&0 \\
0&0
\end{pmatrix},
\begin{pmatrix}
0 \\
0 \\
0 \\
1
\end{pmatrix}
  \right),
\end{equation*}
\begin{equation*}
\mathbb{B}(t)=(B_0(t),B_1(t))=\left(
\begin{pmatrix} 
1&0&0&0 \\
0&t&0&0 \\
0&0&0&0 \\
0&0&0&0
\end{pmatrix},
\begin{pmatrix}
0&0 \\
0&0 \\
1&0 \\
0&t
\end{pmatrix}
  \right),
\end{equation*}
and let $U_i$ be a neighborhood of $A_i$, ($i=0,1,2,3$).  
In the rule of (\ref{li}), $\mathbb{B}(t)$ converges to $\mathbb{A}$ when $t\rightarrow 0$. In terms of (\ref{nbd}), we want to have 
\begin{equation} \label{wa}
\mathbb{B}(t)\in U_{\mathbb{PA}}(U_0,U_1,U_2,U_3)
\end{equation}
when $t<<0$. In fact, (\ref{wa}) holds by the following :
\begin{equation*}
V(\mathbb{B})_0\supseteq V(\mathbb{A})_0, \,\, \overline{B_0(t)|_{V(\mathbb{A})_0}}\in U_0 \text{ since }\Lim\overline{B_0(t)|_{V(\mathbb{A})_0}}=
\begin{pmatrix}
1&0&0&0 \\
0&0&0&0 \\
0&0&0&0 \\
0&0&0&0
\end{pmatrix} ,
\end{equation*}
\begin{equation*}
V(\mathbb{B})_0\supseteq V(\mathbb{A})_1, \,\,\overline{B_0(t)|_{V(\mathbb{A})_1}}\in U_1 \text{ since }\Lim\overline{B_0(t)|_{V(\mathbb{A})_1}}=
\begin{pmatrix}
0&0&0 \\
1&0&0 \\
0&0&0 \\
0&0&0
\end{pmatrix} ,
\end{equation*}
\begin{equation*}
V(\mathbb{B})_1\supseteq V(\mathbb{A})_2, \,\,\overline{B_1(t)|_{V(\mathbb{A})_2}}\in U_2 \text{ since }\Lim\overline{B_1(t)|_{V(\mathbb{A})_2}}=
\begin{pmatrix}
0&0 \\
0&0 \\
1&0 \\
0&0
\end{pmatrix} ,
\end{equation*}
\begin{equation*}
V(\mathbb{B})_1\supseteq V(\mathbb{A})_3, \,\,\overline{B_1(t)|_{V(\mathbb{A})_3}}\in U_3 \text{ since }\Lim\overline{B_1(t)|_{V(\mathbb{A})_3}}=
\begin{pmatrix}
0 \\
0 \\
0 \\
1
\end{pmatrix} .
\end{equation*}
\end{example}
In fact (\ref{nbd}) induces a topology on $\PM$ by the following lemma.
\begin{lemma}[\cite{S1} Lemma 3.2.]
The sets 
\begin{equation*}
\{U_{\mathbb{PA}}(U_0,\dots,U_m)|U_i \text{ is a neighborhood of } \overline{A_i}\,\,(0\le i\le m)\}
\end{equation*}
satisfy the axiom of a base of neighborhoods of $\mathbb{PA}$, and hence define a topology in $\PM$.
\end{lemma}
Moreover the following theorem holds.
\begin{theorem}[\cite{S1} Theorem 5.1., Proposition 3.9, 3.10.]
The set $\PM$ is compact, and $\PGL(V)$ is dense open in $\PM$.
\end{theorem}
Here we regard an element of $\PGL(V)$ as a one-term element of $\PM$, and $\PM$ is a compactification of $\PGL(V)$. 
\subsubsection{Monoid structure of $\mathbb{P}\widetilde{M}$}
For $\mathbb{A}=(A_0,A_1,\dots,A_m)$, $\mathbb{B}=(B_0,B_1,\dots,B_n)$$\in\mathbb{P}\widetilde{M}$, define $\mathbb{AB}$ by removing the redundant matrices from
\begin{equation} \label{prod}
\begin{split}
\mathbb{AB}=(A_0B_0,A_1B_0,\dots,&A_mB_0,A_0B_1, \\
&\dots,A_mB_1,\dots,A_0B_n,\dots,A_mB_n).
\end{split}
\end{equation}
This defines a monoid structure on $\mathbb{P}\widetilde{M}$ (\cite{S1} Proposition 6.6.).
\subsection{$\mathbb{P}M$-monoids}
We next review a $\mathbb{P}M$-monoid defined in \cite{M}. 
Let $T$ be a maximal torus of $\PGL_n$. Then we consider the following monoid 
\begin{equation*}
\mathscr{R}_n=\overline{N_{\PGL_n}(T)}/T,
\end{equation*}
where the closure is taken in the topology of $\mathbb{P}M$. We call this monoid $\mathscr{R}_n$ a $\mathbb{P}M$-monoid. The structure of a $\PM$-monoid can be described in terms of a matched pairs. We first explain a matched pairs (cf. \cite{MS},\cite{T}). Let $S$ be a monoid. We denote the unit element of $S$ by $1_S$.
\begin{definition}
Let $S, B$ be monoids which have binary operations $\rightharpoonup:S\times B\rightarrow B$ and $\leftharpoonup:S\times B\rightarrow S$. A matched pair of monoids means a triple $(S,B,\sigma)$, where $S,B$ are monoids and
\begin{equation*}
\sigma:S\times B\rightarrow B\times S, \, (s,b)\mapsto (s\rightharpoonup b,s\leftharpoonup b)
\end{equation*}
is a map satisfying the following conditions : 
\begin{enumerate}
\item[(1)] \label{m1}
$s\rightharpoonup(t\rightharpoonup b)=st\rightharpoonup b$,
\item[(2)] \label{m2}
$st\leftharpoonup b=(s\leftharpoonup(t\rightharpoonup b))(t\leftharpoonup b)$,
\item[(3)] \label{m3}
$(s\leftharpoonup b)\leftharpoonup c=s\leftharpoonup bc$,
\item[(4)] \label{m4}
$s\rightharpoonup bc=(s\rightharpoonup b)((s\leftharpoonup b)\rightharpoonup c)$,
\item[(5)] \label{m5}
$1_S\rightharpoonup b=b$,
\item[(6)] \label{m6}
$s\rightharpoonup 1_B=1_B$,
\item[(7)] \label{m7}
$s\leftharpoonup 1_B=s$,
\item[(8)] \label{m8}
$1_S\leftharpoonup b=1_S$
\end{enumerate}
for $s,t\in S$, $b,c\in B$.
\end{definition}
The product $B\times S$ forms a monoid with product
\begin{equation*}
(b,s)(c,t)=(b(s\rightharpoonup c),(s\leftharpoonup c)t).
\end{equation*}
This monoid is denoted by $B\Join_{\sigma} S$. \\
Let
\begin{equation*}
\begin{split}
P_n=\Biggl\{(\{i_1,\dots,i_{k_1}\},\{i_{k_1+1},\dots,i_{k_2}\},&\dots,\{i_{k_m+1},\dots,i_n\}) \\
&\mid \begin{matrix}\{i_1\dots,i_n\}=\{1,\dots,n\} \\ 1\le k_1<k_2<\dots<k_{m-1}< n\end{matrix} \Biggr\}.
\end{split}
\end{equation*}
An element of $P_n$ is called an ordered set partitions of $[n]:=\{1,2,\dots,n\}$. The set $P_n$ has a monoid structure defined by
\begin{equation*}
(p_1,\dots,p_m)*(p_1^{\prime},\dots,p^{\prime}_{m^{\prime}}):=(p_1\cap p_1^{\prime},\dots,p_m\cap p_1^{\prime},\dots,p_1\cap p^{\prime}_{m^{\prime}},\dots,p_m\cap p^{\prime}_{m^{\prime}}).
\end{equation*}
Then the following proposition holds.
\begin{proposition}[\cite{M} Proposition 3.5] \label{rm}
Let $\mathscr{R}_n$ be the $\PM$-monoid, $S_n$ the symmetric group and $P_n$ the collection of the ordered set partitions of $[n]$. Define a map
\begin{equation*}
\varphi:P_n\times S_n\rightarrow S_n\times P_n,\, ((p_1,\dots,p_m),w)\mapsto (w,(w^{-1}(p_1),\dots,w^{-1}(p_m))).
\end{equation*}
Then
\begin{equation*}
\mathscr{R}_n\simeq S_n\Join_{\varphi} P_n.
\end{equation*}
\end{proposition}
The $\mathbb{P}M$-monoid has a presentation by generators and relations. We first define some notations. We denote
\begin{equation} \label{par}
(k_1,\dots,k_{m-1})=(\{1,\dots,k_1\},\dots,\{k_{m-1}+1,\dots,n\}).
\end{equation} 
For $i=1,\dots,n-1$ and a partition $(k_1,\dots,k_{m-1})$ (cf. (\ref{par})), if there exists $j\in \{1,\dots,n\}$ such that $\{i,i+1\}\subseteq\{k_{j-1}+1,\dots,k_j\}$, then we set 
\begin{equation*}
i_{*}:=j.
\end{equation*} 
For $\sigma\in S_n$ we define a map $\varphi_{\sigma}:P_n\rightarrow P_n$ by
\begin{equation*}
(p_1,\dots,p_m)\mapsto (\sigma^{-1}(p_1),\dots,\sigma^{-1}(p_m)).
\end{equation*}
We define a set 
\begin{equation*}
\Pi_n=\{(k_1,\dots,k_{m-1}):1\le k_1<\dots<k_{m-1}< n\},
\end{equation*}
where $(k_1,\dots,k_{m-1})$ is $(\ref{par})$.
For $p\in P_n$, take an element $w\in S_n$ such that $wpw^{-1}\in \Pi_n$, and set
\begin{equation*}
u^{w}(p):=wpw^{-1}\in \Pi_n.
\end{equation*}
We also set
\begin{equation*}
{\rm{Ad}}(\sigma)(e):=\sigma^{-1}e\sigma.
\end{equation*}
Using these notations we obtain the following monoid presentation of the $\PM$-monoid $\mathscr{R}_n$.
\begin{proposition}[\cite{M} Proposition 3.8] \label{rep}
The $\PM$-monoid $\mathscr{R}_n$ has a monoid presentation with generating set
\begin{equation*}
\{s_1,\dots,s_{n-1},e_{k_1,\dots,k_{m-1}}\,\, (1\le k_1<\dots<k_{m-1}< n)\}
\end{equation*}
and defining relations
\begin{align}
s_i^2&=1  &\,\,(&1\le i\le n-1),  \label{re1} \\
s_is_j&=s_js_i  &(&1\le i,j \le n-1,\,|i-j|\ge 2), \label{re2}   \\
s_is_{i+1}s_i&=s_{i+1}s_is_{i+1}  &(&1\le i\le n-1), \label{re3} 
\end{align}
\begin{align}
&\,\,\,\,\,\,\,\,\,\,\,\,\,\,\,\,\,\,\,\,\,\,\,\,\,\,\,\,\,\,\,\,\,\, e_{k_1,\dots,k_{i_{*}},\dots,k_{m-1}}s_i=s_ie_{k_1,\dots,k_{i_{*}},\dots,k_{m-1}} \label{re4}  \\ 
&\,\,\,\,\,\,\,\,\,\,\,\,\,\,\,\,\,\,\,\,\,\,\,\,\,\,\,\,\,\,\,\,\,\,\,\,\,\,\,\,\,\,\,\,\,\,\,\,\,\,\,\, \begin{pmatrix}
1\le i\le n-1 \\
1\le k_1<\dots<k_{i_{*}}<\dots<k_{m-1}< n
\end{pmatrix}, \notag \\
&\,\,\,\,\,\,\, e_{k_1,\dots,k_{m-1}}s_{i_1}\dots s_{i_r}e_{l_1,\dots,l_{m^{\prime}-1}}
={\rm{Ad}}(s_{j_1}\dots s_{j_t})(e_{q})s_{i_1}\dots s_{i_r}  \label{re5} \\
&\,\,\,\,\,\,\,\,\,\,\,\,\,\,\,\,\,\,\,\,\,\,\,\,\,\,\,\,\,\,\,\,\,\,\,\,\,\,\,\,\,\,\,\,\,\,\,\,\,\,\,\, \begin{pmatrix}
1\le k_1<\dots<k_{m-1}< n \\
1\le l_1<\dots<l_{m^{\prime}-1}< n \\
\{i_1,i_1+1\}\nsubseteq \{k_{l-1}+1,\dots,k_l\},  ^{\forall}l=1,\dots,n \\
q=u^{s_{j_1}\dots s_{j_t}}((k_1,\dots,k_{m-1})*\varphi_{(s_{i_1}\dots s_{i_r})^{-1}}((l_1,\dots,l_{m^{\prime}-1}))) \notag
\end{pmatrix}.
\end{align}
\end{proposition}  
\subsection{Braid $\mathbb{P}M$-monoids}
We recall a braid monoid defined in \cite{M}. The notations are the same as those in Proposition \ref{rep}, and we add the following notation. We denote by $b|_I$ an element of  braid group of  $\#I$-strings which has strings in the place of any $i\in I$ for $b\in B_n$ and $I\subset\{1,\dots,n\}$.  If $s_{i_1},\dots,s_{i_r}\in B_n$ satisfy $s_{i_1}\dots s_{i_r}|_I=id|_I$, where $I\subset\{1,\dots,n\}$ and $id$ is the identity braid in $B_n$, then we abbreviate this condition as $\{i_1,\dots,i_r\}|_I=id$.
\begin{definition} \label{Bpm}
The braid $\mathbb{P}M$-monoid is a monoid which is defined by the monoid presentation with generating set 
\begin{equation*}
\{s_1^{\pm 1},\dots,s_{n-1}^{\pm 1},e_{k_1,\dots,k_{m-1}}\,\,(1\le k_1<\dots<k_{m-1}< n)\}
\end{equation*}
and defining relations
\begin{align} 
s_is_i^{-1}&=s_i^{-1}s_i=1 \,\, &(&1\le i\le n-1), \label{re1-} \\
s_is_j&=s_js_i ,\,\ &(&1\le i,j \le n-1,\,|i-j|\ge 2), \label{re2-} \\
s_is_{i+1}s_i&=s_{i+1}s_is_{i+1} \,\, &(&1\le i\le n-1), \label{re3-}
\end{align}
\begin{align}
&\,\,\,\,\,\,\,\,\,\,\,\,\,\,\,\,\,\,\,\,\,\,\, s_{i_1}^{\pm 1}\dots s_{i_r}^{\pm 1}e_{k_1,\dots,k_{m-1}}s_{j_1}^{\pm 1}\dots s_{j_t}^{\pm 1}=e_{k_1,\dots,k_{m-1}}   \label{re4-} \\
&\,\,\,\,\,\,\,\,\,\,\,\,\,\,\,\,\,\,\,\,\,\,\,\,\,\,\,\,\,\,\,\,\,\,\,\,\,\,\,\,\,\,\,\,\,\,\,\,\,\,\,\,\begin{pmatrix}
\{i_1,\dots,i_r,j_1,\dots,j_t\}|_{\{k_{j-1}-1,\dots,k_j\}}=id \\ ^{\forall}j=1,\dots ,m 
\end{pmatrix}, \notag \\
&\,\,\,\,\,\,\,\,\,\,\,\,\,\,\,\,\,\,\,\,\,\,\,\,\,\,\, e_{k_1,\dots,k_{m-1}}s_{i_1}^{\pm 1}\dots s_{i_r}^{\pm 1}e_{l_1,\dots,l_{m^{\prime}-1}} 
={\rm{Ad}}(s_{j_1}^{\pm 1}\dots s_{j_t}^{\pm 1})(e_{q})s_{i_1}^{\pm 1}\dots s_{i_r}^{\pm 1} \label{re5-} \\
&\,\,\,\,\,\,\,\,\,\,\,\,\,\,\,\,\,\,\,\,\,\,\,\,\,\,\,\,\,\,\,\,\,\,\,\,\,\,\,\,\,\,\,\,\,\,\,\,\,\,\,\,\begin{pmatrix} 
1 \le k_1<\dots<k_{m-1}< n \\
1\le l_1<\dots<l_{m^{\prime}-1}< n  \\
\{i_1,i_1+1\}\nsubseteq \{k_{l-1}+1,\dots,k_l\}, ^{\forall} l=1,\dots,m \\
q=u^{s_{j_1}^{\pm 1}\dots s_{j_t}^{\pm 1}}((k_1,\dots,k_{m-1})*\varphi_{(s_{i_1}\dots s_{i_r})^{-1}}((l_1,\dots,l_{m^{\prime}-1})))
\end{pmatrix}. \notag
\end{align}
\end{definition}
The braid $\mathbb{P}M$-monoid is described by a geometric braid. We denote by $\mathscr{M}$ the monoid defined in Definition \ref{Bpm}. To describe the monoid $\mathscr{M}$ geometrically we shall define a $\mathbb{P}M$-braid.  

First, we shall define an arc.
\begin{definition}
An arc is the image of an embedding from the unit interval $[0,1]$ into $\mathbb{R}^3$. 
\end{definition}

Take the usual coordinate system $(x,y,z)$ for $\mathbb{R}^3$. Choose $z_0^{(m)}<z_1^{(m)}<\dots<z_0^{(2)}<z_1^{(2)}<z_0^{(1)}<z_1^{(1)}$. Mark $n\geq 0$ distinct points $P_1^i,\dots,P_n^i$ on a line in the plane $z=z_1^{(i)}$, and project this orthogonally on the plane $z=z_0^{(i)}$,  yielding points $Q_1^i,\dots,Q_n^i$ for each $i=1,\dots,m$. 

A $\mathbb{P}M$-braid on $n$ strings is a system
\begin{equation*}
\beta=\{\beta_1,\dots,\beta_{k_1},\beta_{k_1+1},\dots,\beta_{k_2},\beta_{k_2+1},\dots,\beta_{k_{m-1}+1},\dots,\beta_n\}
\end{equation*}
of $n$ arcs for some $1\le k_1<k_2<\dots<k_{m-1}<n$ such that 
\begin{enumerate}
\item[(1)] There exists a partial one-one mapping of rank $k_1$ 
\begin{equation*}
\Phi_1^{\beta}:\{1,\dots,n\}\rightarrow \{1,\dots,n\}
\end{equation*}
with domain $\{i_1,\dots,i_{k_1}\}$ such that $\beta_j$ connects $P_{i_j}^1$ to $Q_{\Phi_1^{\beta}(i_j)}^1$ for $j=1,\dots,k_1$.\\
There exists a partial one-one mapping of rank $k_2-k_1$ 
\begin{equation*}
\Phi_2^{\beta}:\{1,\dots,n\}\backslash\{i_1,\dots,i_{k_1}\}\rightarrow \{1,\dots,n\}\backslash\{i_1,\dots,i_{k_1}\}
\end{equation*}
with domain $\{i_{k_1+1},\dots,i_{k_2}\}$ such that $\beta_j$ connects $P_{i_j}^2$ to $Q_{\Phi_1^{\beta}(i_j)}^2$ for $j=k_1+1,\dots,k_2$.
\begin{center}
$\dots\dots$
\end{center}
There exists a partial one-one mapping of rank $n-k_{m-1}$
\begin{equation*}
\Phi_m^{\beta}:\{1,\dots,n\}\backslash\{i_1,\dots,i_{k_{m-1}}\}\rightarrow \{1,\dots,n\}\backslash\{i_1,\dots,i_{k_{m-1}}\}
\end{equation*}
with domain $\{i_{k_{m-1}+1},\dots,i_n\}$ such that $\beta_j$ connects $P_{i_j}^m$ to $Q_{\Phi_m^{\beta}(i_j)}^m$ for $j=k_{m-1}+1,\dots,n$. \\
\item[(2)] For $j=1,\dots,m$, the arc $\beta_l$ intersects the plane $z=z_0^{(j)}$ exactly once, and $\beta_l$ intersects the plane $z=z_1^{(j)}$ exactly once, for $l=k_{j-1}+1,\dots,k_j$, and $\beta_s$ does not intersect $z=z_0^{(t)}$, $z=z_1^{(t)}$ for $s\neq t$. \\
\item[(3)] For $j=1,\dots,m$ the union $\beta_{k_{j-1}+1}\cup\dots\cup\beta_{k_j}$ of the arcs intersects each of parallel planes $z=z_0^{(j)},z=z_1^{(j)}$ at exactly $k_j-k_{j-1}$ distinct points. 
\end{enumerate}
\begin{example}
The following is a $\PM$-braid. \\
\raisebox{50pt}{$\beta=$}
\includegraphics[width=4cm,bb=0 0 320 300]{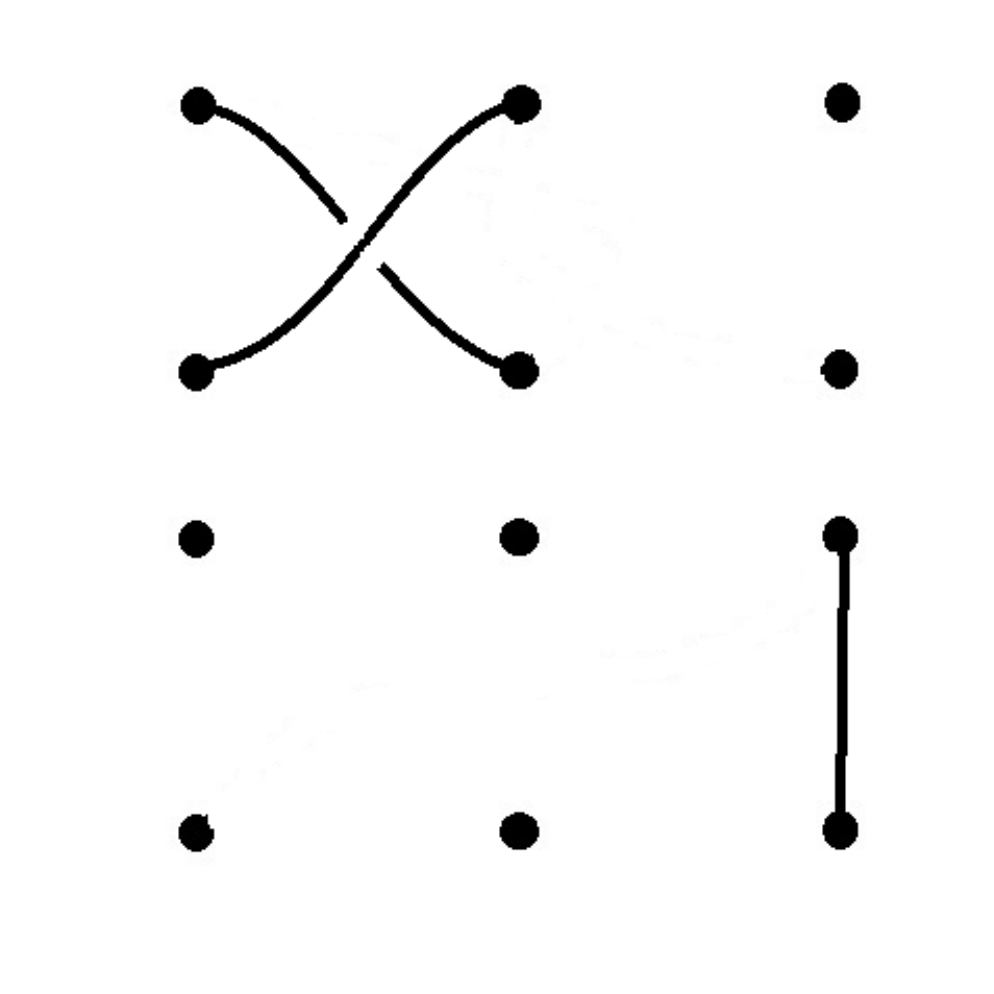} 
\end{example}

Two $\mathbb{P}M$-braids
\begin{equation*}
\begin{split}
 &\beta=\{\beta_1,\dots,\beta_{k_1},\beta_{k_1+1},\dots,\beta_{k_2},\beta_{k_2+1},\dots,\beta_{k_{m-1}+1},\dots,\beta_n\}, \\
 &\gamma=\{\gamma_1,\dots,\gamma_{k_1},\gamma_{k_1+1},\dots,\gamma_{k_2},\gamma_{k_2+1},\dots,\gamma_{k_{m^{\prime}-1}+1},\dots,\gamma_n\}
 \end{split}
 \end{equation*}
are defined to be equivalent if \\
\begin{itemize}
\item[(1)] $m=m^{\prime}$ and $\Phi_i^{\beta}=\Phi_i^{\gamma}$ for $i=1,\dots,m$, \\
\item[(2)] $\beta$ and $\gamma$ are homotopy equivalent, i.e., there exist continuous maps 
\begin{equation*}
F_j:[0,1]\times [0,1]\rightarrow \mathbb{R}^3,\,\,\,\,\,\,\,\,\,\,\, (j=1,\dots,m)
\end{equation*}
such that for all $s, t\in [0,1]$,
\begin{align*}
&\begin{matrix}
F_j(t,0)&=&\beta_j(t)  \\
F_j(t,1)&=&\gamma_j(t)
\end{matrix}
&(&j=1,\dots ,m), \\
&\begin{matrix}
F_j(0,s)&=&P_{i_j}^1  \\
F_j(1,s)&=&Q_{\Phi_1^{\beta}(i_j)}^1
\end{matrix}
&(&j=1,\dots ,k_1), \\
&\begin{matrix}
F_j(0,s)&=&P_{i_j}^2 \\
F_j(1,s)&=&Q_{\Phi_2^{\beta}(i_j)}^2
\end{matrix}
&(&j=k_1+1,\dots,k_2),  \\
&&\dots \\
&\begin{matrix}
F_j(0,s)&=&P_{i_j}^m  \\
F_j(1,s)&=&P_{\Phi_m^{\beta}(i_j)}^m
\end{matrix}
&(&j=k_{m-1}+1,\dots,n), 
\end{align*}
and, for each $s\in[0,1]$ if we define
\begin{equation*}
\beta^s=\{\beta_1^s,\dots,\beta_m^s\},
\end{equation*}
where
\begin{equation*}
\beta_j^s(t)=F_j(s,t) {\text{ for }}j=1,\dots,n,
\end{equation*}
then $\beta^s$ is itself a $\mathbb{P}M$-braid. 
\end{itemize}
\begin{example}
The following $\PM$-braids are equivalent. \\
\raisebox{40pt}{$\beta=$}
\includegraphics[width=4cm,bb=0 0 320 360]{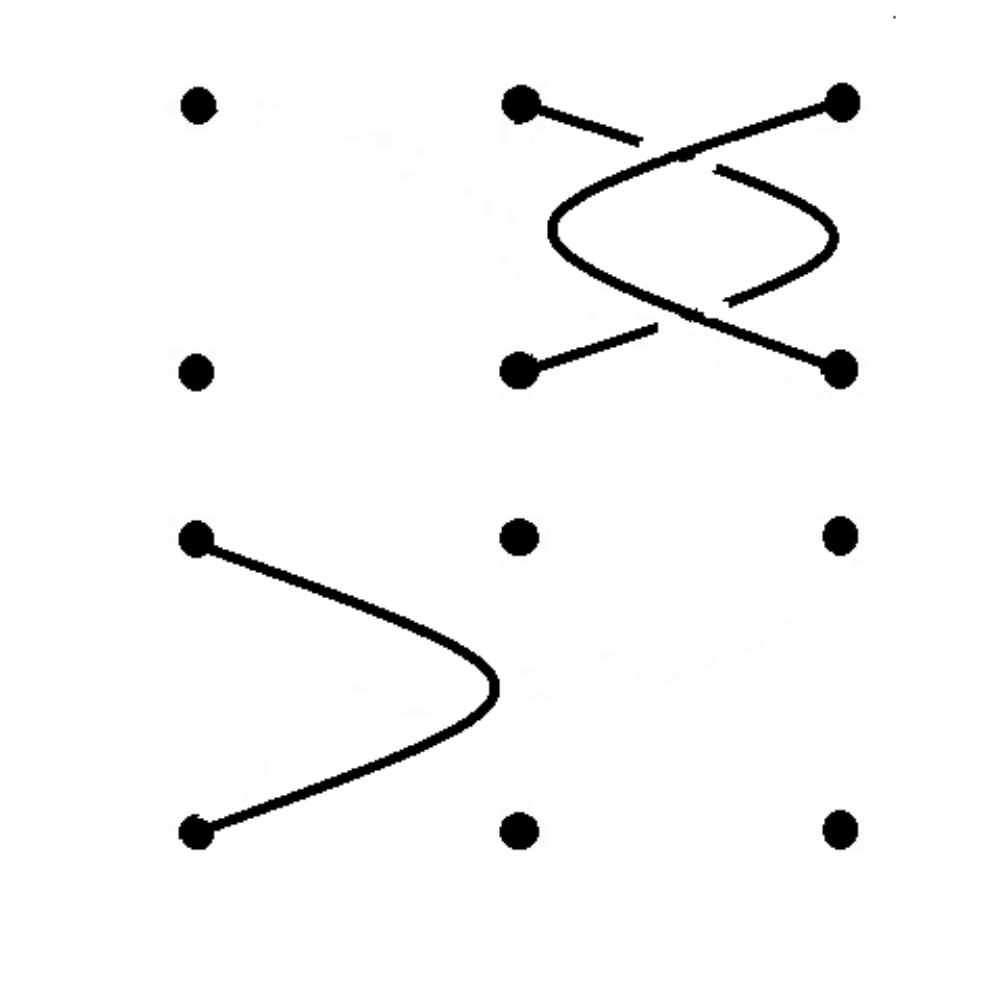} 
\raisebox{40pt}{$\gamma=$}
\includegraphics[width=4cm,bb=0 0 320 360]{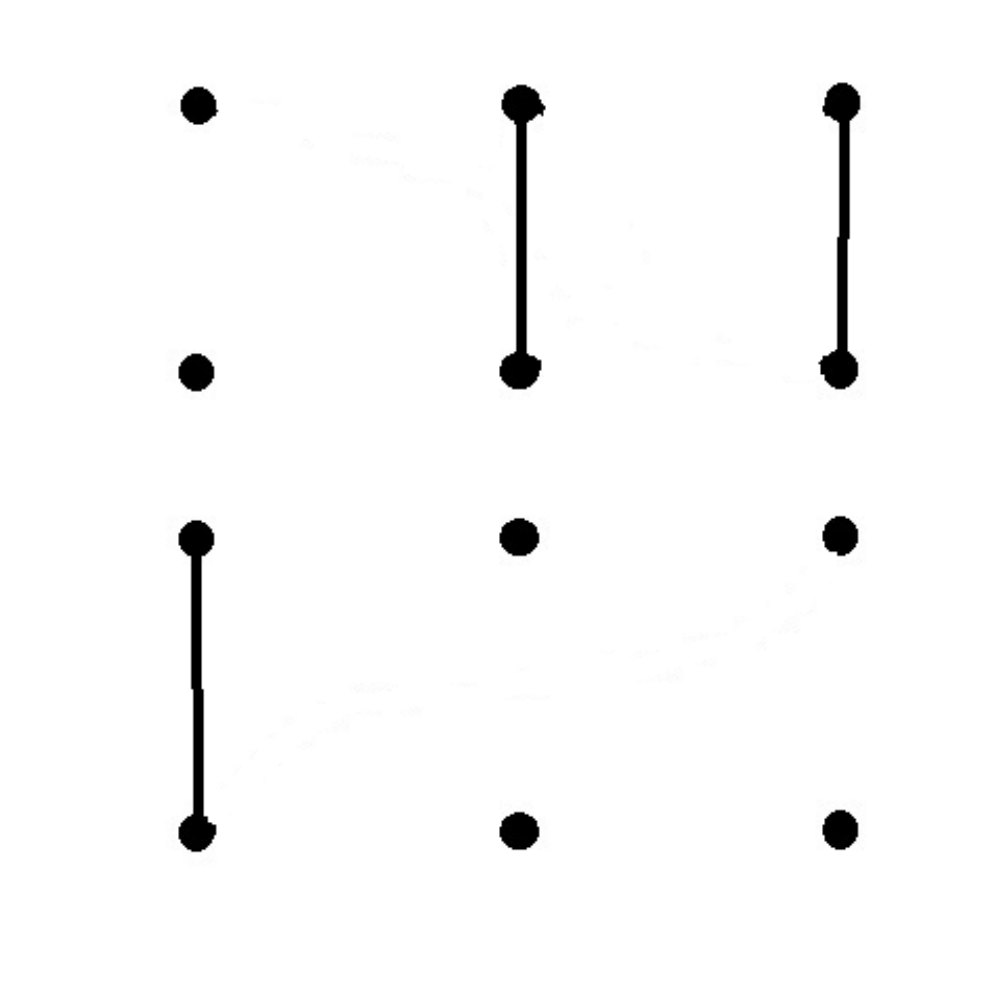} \\
\end{example}

Define the product $\beta\gamma$ of two braids
\begin{equation*}
\begin{split}
 &\beta=\{\beta_1,\dots,\beta_{k_1},\beta_{k_1+1},\dots,\beta_{k_2},\beta_{k_2+1},\dots,\beta_{k_{m-1}+1},\dots,\beta_n\}, \\
 &\gamma=\{\gamma_1,\dots,\gamma_{k_1},\gamma_{k_1+1},\dots,\gamma_{k_2},\gamma_{k_2+1},\dots,\gamma_{k_{m^{\prime}-1}+1},\dots,\gamma_n\}
 \end{split}
\end{equation*}
as follows. 

We first define an operation ($k_il_j$). Take $z_1^{(11)}>z_0^{(11)}>z_1^{(21)}>z_0^{(21)}>\dots>z_1^{(m1)}>z_0^{(m1)}>z_1^{(12)}>z_0^{(12)}>\dots>z_1^{(m2)}>z_0^{(m2)}>
\dots>z_1^{(mm^{\prime})}>z_0^{(mm^{\prime})}$.\\

($k_il_j$) : 
\begin{itemize}
\item[(1)] Translate $\{\gamma_{l_{j-1}+1},\dots,\gamma_{l_j}\}$ parallel to itself so that the upper plane of  $\{\gamma_{l_{j-1}+1},\dots,\gamma_{l_j}\}$ coincides with the lower plane of $\{\beta_{k_{i-1}+1},\dots,\beta_{k_j}\}$; 
\item[(2)] Translate the above system of arcs so that the upper plane of  $\{\beta_{k_{i-1}+1},\dots,\beta_{k_j}\}$ coincides with $z=z_1^{(ij)}$. Keeping   $z=z_1^{(ij)}$ fixed, contract the resulting systems of arcs so that the translated lower plane of $\{\gamma_{l_{j-1}+1},\dots,\gamma_{l_j}\}$ lies into the position of $z=z_0^{(ij)}$; 
\item[(3)] Remove any arc that do not now join the upper plane to the lower plane. 
\end{itemize}

Then take the operations ($k_1l_1$),$\dots$,($k_ml_1$),($k_1l_2$),$\dots$,($k_ml_2$),($k_1l_{m^{\prime}}$),$\dots$,($k_ml_{m^{\prime}}$), finally remove empty systems of arcs. The resulting $\mathbb{P}M$-braid is denoted by $\beta\gamma$.

\begin{example}
Let $\beta$ and $\gamma$ be the following $\mathbb{P}M$-braids: \\
\raisebox{40pt}{$\beta=\begin{matrix}\beta_1 \\ \beta_2 \end{matrix}=$}
\includegraphics[width=4cm,bb=0 0 320 250]{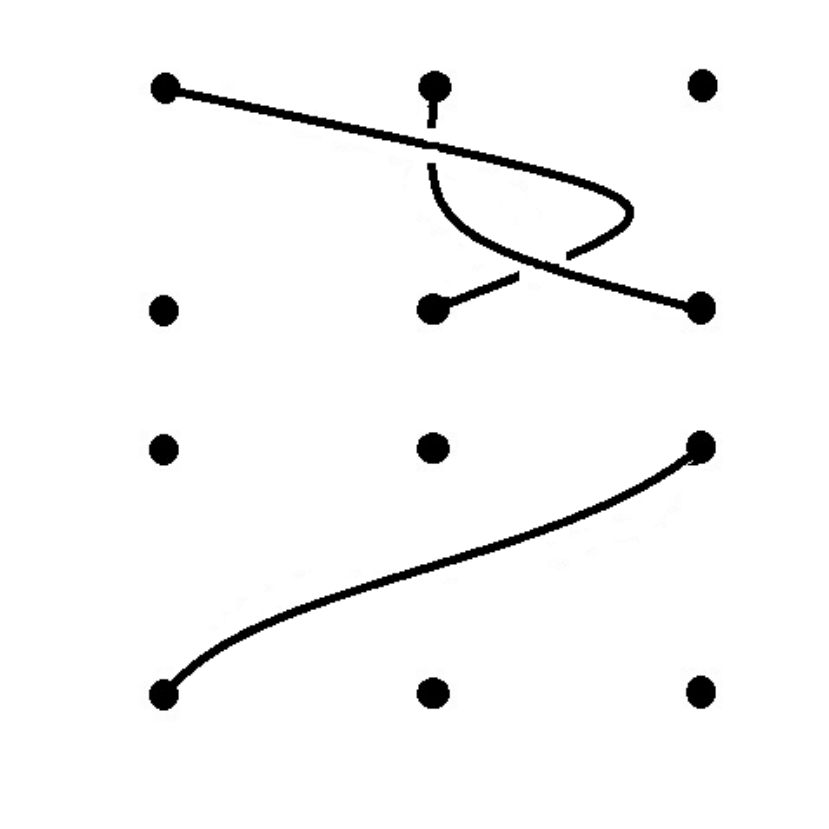} 
\raisebox{40pt}{$\gamma=\begin{matrix}\gamma_1 \\ \gamma_2 \end{matrix}=$}
\includegraphics[width=4cm,bb=0 0 320 250]{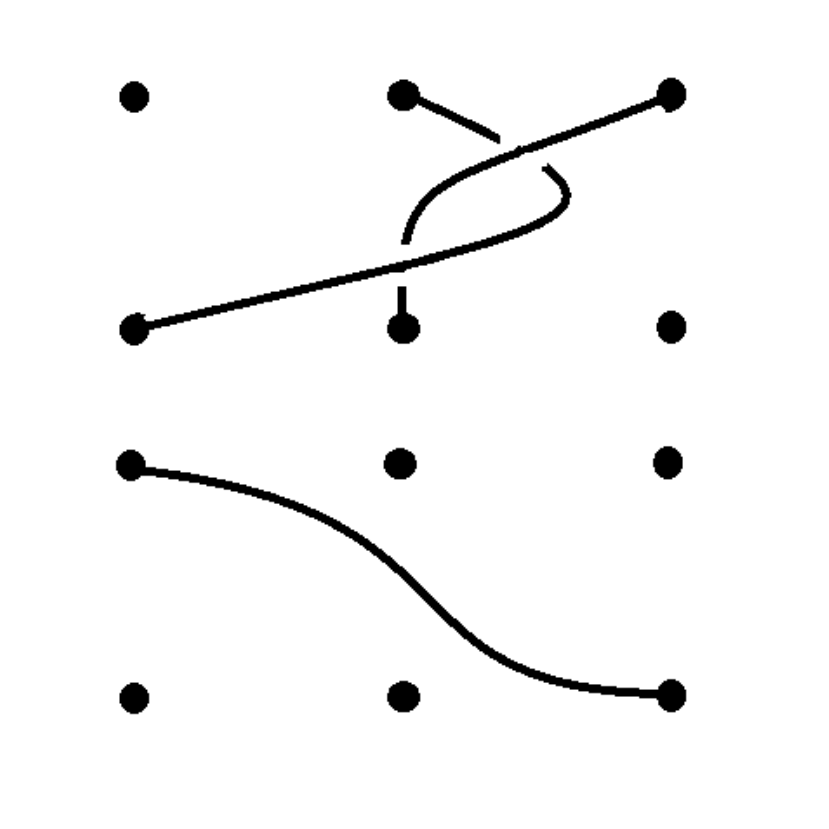} \\
then $\beta^2,\beta\gamma,\gamma\beta$ are obtained as follows: \\
\raisebox{140pt}{$\beta^2=\begin{matrix} \beta_1\beta_1 \\ \beta_2\beta_1 \\ \beta_1\beta_2 \\ \beta_2\beta_2 \end{matrix}=$}
\includegraphics[width=4cm,bb=0 0 320 450]{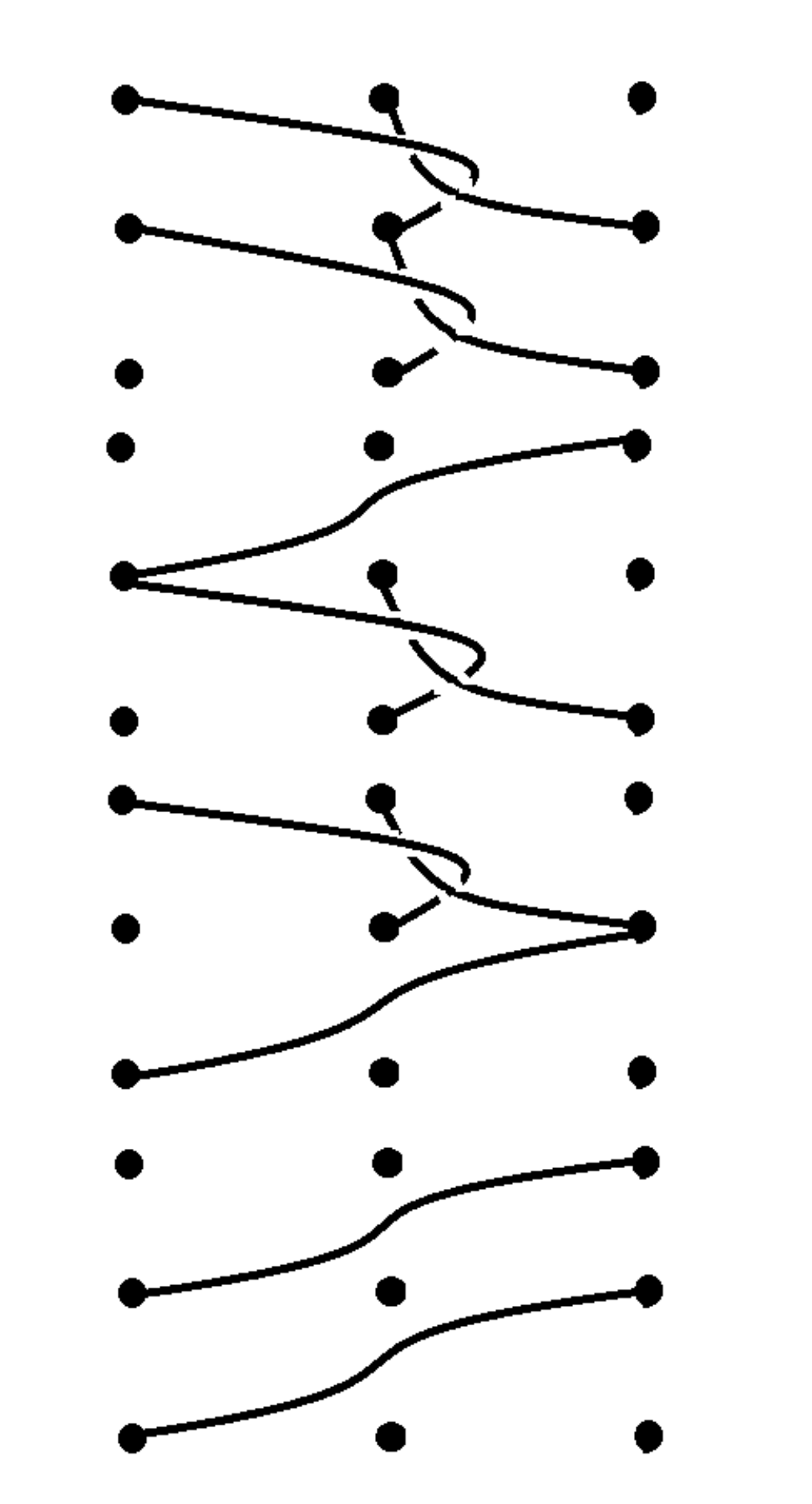}
\raisebox{140pt}{$=$}
\raisebox{70pt}{\includegraphics[width=4cm,bb=0 0 320 430]{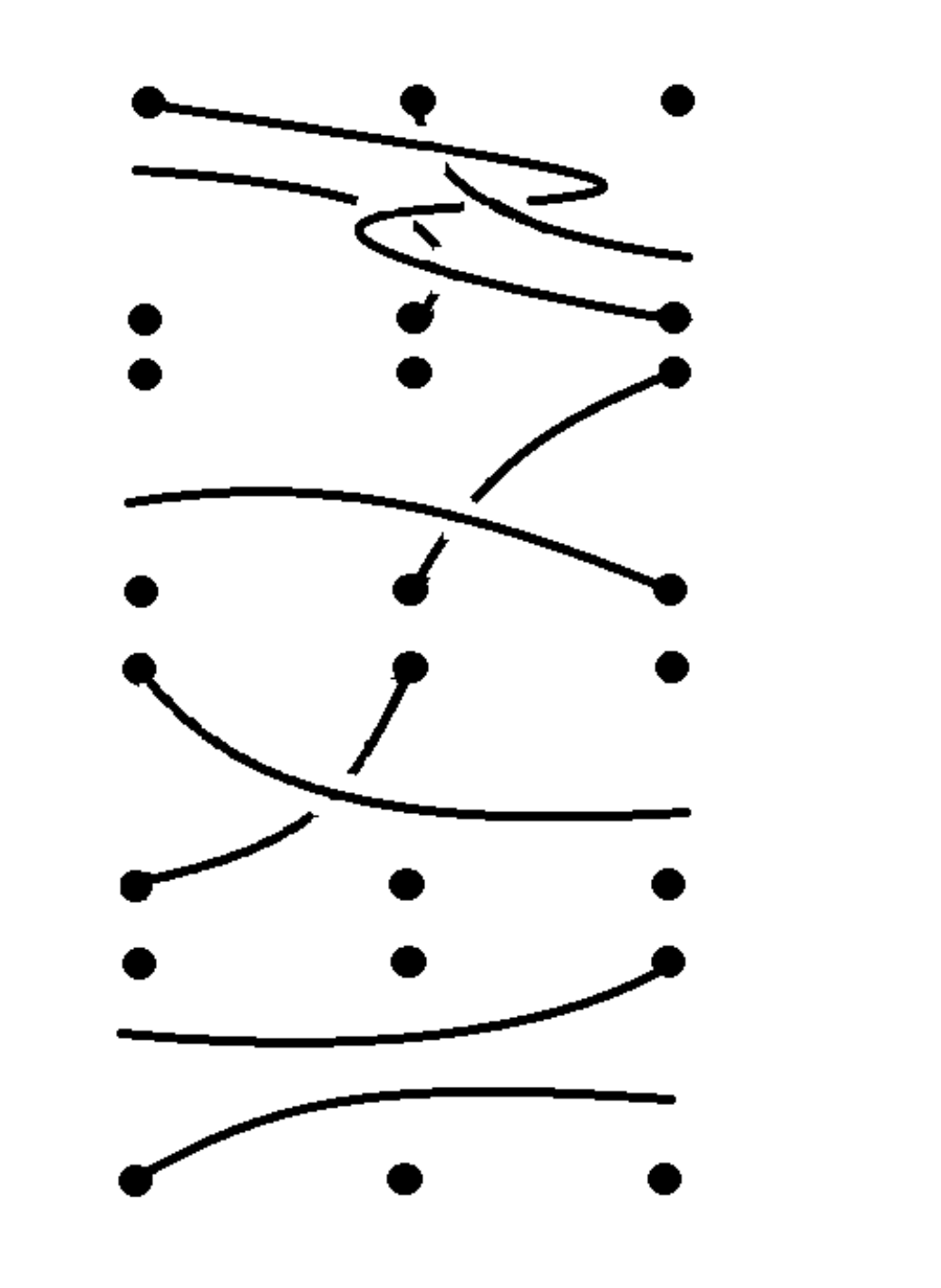}}
\raisebox{140pt}{$=$}
\raisebox{70pt}{\includegraphics[width=4cm,bb=0 0 320 350]{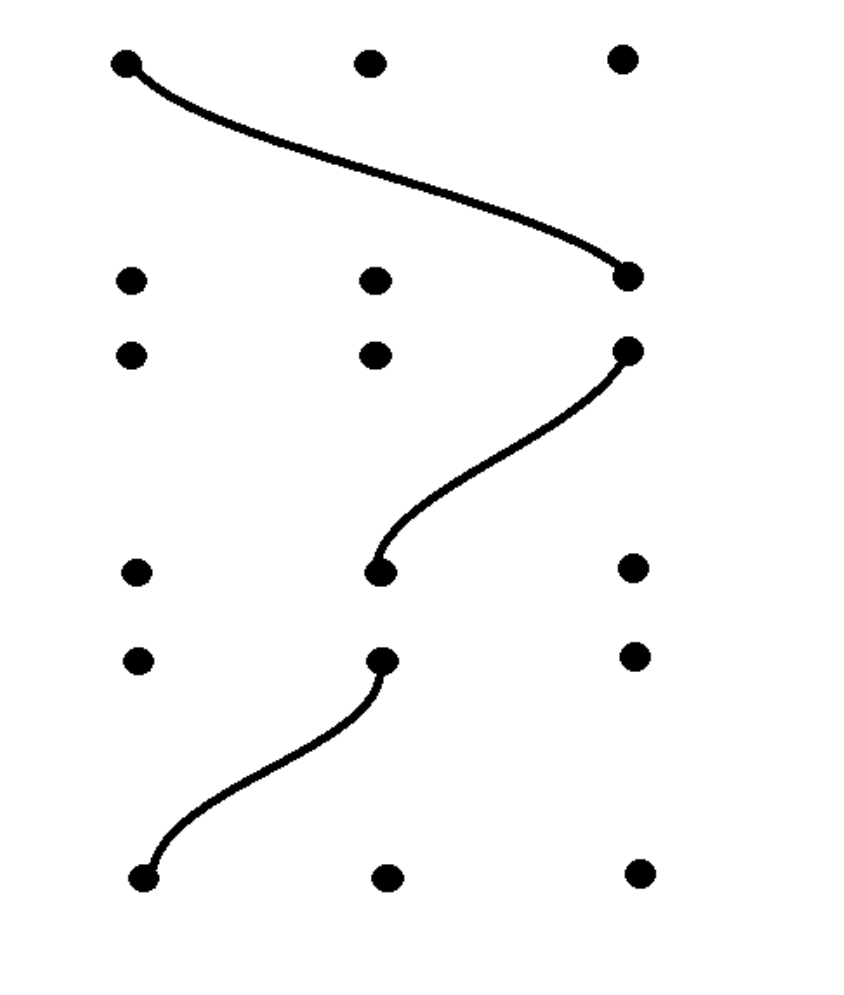}} \\
\raisebox{80pt}{$\beta\gamma=$}
\raisebox{50pt}{\includegraphics[width=4cm,bb=0 0 320 200]{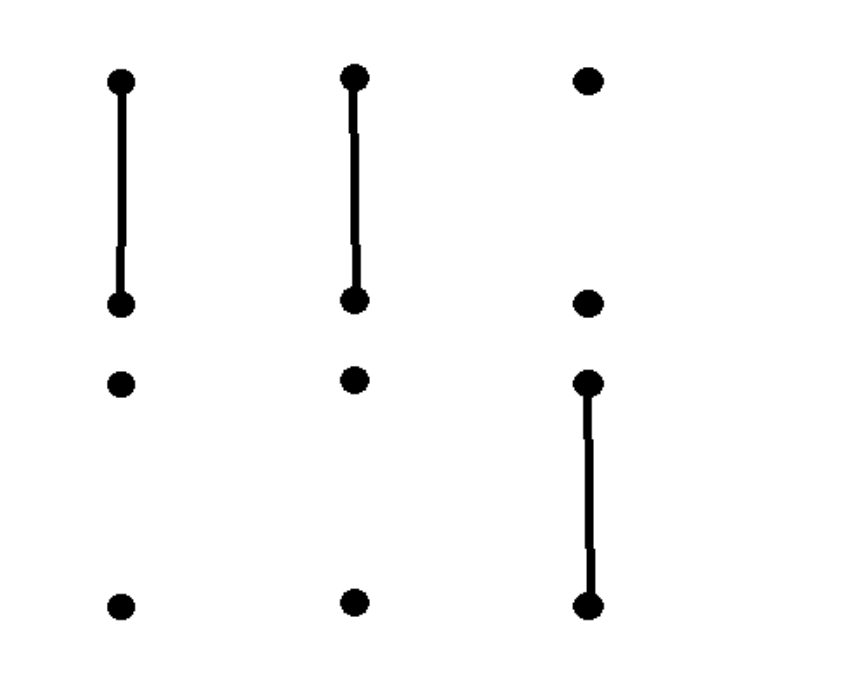}} 
\raisebox{80pt}{$\gamma\beta=$}
\raisebox{50pt}{\includegraphics[width=4cm,bb=0 0 320 200]{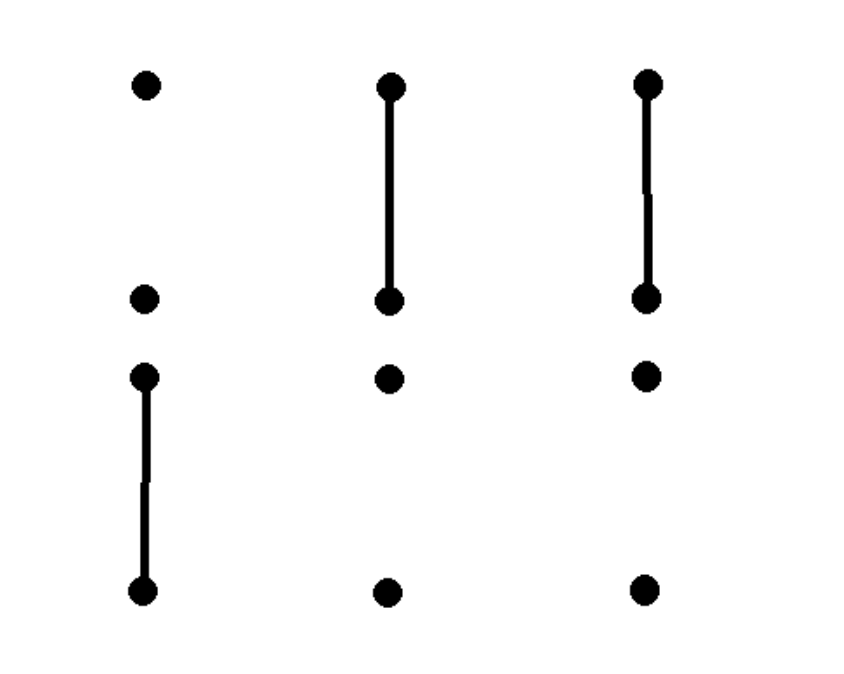}} 
\end{example}
We denote by $[\beta]$ the homotopy equivalence class of $\beta$. Put
\begin{equation*}
\mathscr{RB}_n=\{[\beta]:\beta {\text{ is a }} \mathbb{P}M{\text{-braid}}\}.
\end{equation*}
\begin{theorem}[\cite{M} Theorem 4.5] \label{pmb}
The braid $\mathbb{P}M$-monoid $\mathscr{M}$ is isomorphic to the monoid $\mathscr{RB}_n$.
\end{theorem}

\section{Mapping class groups and inverse mapping class monoids}
We review the concepts of mapping class groups and inverse mapping class monoids. Let $S=S_{g,b,n}$ be a connected orientable surface of genus $g$ with $b$ boundary components and the set of $n$ marked points $p=\{1,2,\dots,n\}$.
\subsection{Mapping class groups}
Let ${\rm{Homeo}}^+(S,\partial S)$ denote the group of orientation-preserving homeomorphisms of $S$ that restrict to the identity on $\partial S$ and fix $p$ as a set. Then define mapping class group  of $S$ denoted by $\mathcal{M}(S)$ as
\begin{equation*}
\mathcal{M}(S)={\rm{Homeo}}^+(S,\partial S)/{\text{isotopy}}.
\end{equation*}
We describe some examples of mapping class groups.
\begin{lemma}[Alexander lemma \cite{FM} Lemma 2.1.]
Let $D^2$ be a disk. Then $\mathcal{M}(D^2)$ is trivial.
\end{lemma}
\begin{proposition}[\cite{FM} Proposition 2.3.]
Let $S^2$ be a sphere. Then
\begin{equation*}
\mathcal{M}(S^2\backslash\{1,2,3\})\simeq S_3.
\end{equation*}
\end{proposition}
\begin{proposition}[\cite{FM} Section 9.1.3.]
Let $B_n$ be braid group of $n$ strings. Then 
\begin{equation*}
\mathcal{M}(D^2\backslash\{1,2,\dots,n\})\simeq B_n.
\end{equation*}
\end{proposition}
\subsection{Inverse mapping class monoids}
Next we review inverse mapping class monoid defined by V. V. Vershinin \cite{V2}. Let $S=S_{g,b,n}$ be the same surface as above. Let $f$ be a orientation-preserving homeomorphism of $S$ such that 
\begin{equation*}
f(\{i_1,\dots,i_k\})=\{j_1,\dots,j_k\},
\end{equation*} 
where $k\le n$. Denote the set of isotopy classes of such maps by $\mathcal{IM}(S)$ and is called inverse mapping class monoid. The same way let $h$ be a homeomorphism of $S$ which maps $l$ points, $l\le n$, say $\{s_1,\dots,s_l\}$ to $l$ points. Consider the intersection of the sets $\{j_1,\dots,j_k\}$ and $\{s_1,\dots,s_l\}$, let it be the set of cardinality $m$ ($m\le k$), it may be empty. Then the composition of $f$ and $h$ maps $m$ points to $m$ points (may be different).
\begin{proposition}[\cite{V2} Section 6]
Let $S^2$ be a sphere. Then
\begin{equation*}
\mathcal{IM}(S^2\backslash\{1,2,3\})\simeq R_3.
\end{equation*}
\end{proposition}
The inverse braid monoid $IB_n$ was constructed by D. Easdown and T. G. Lavers \cite{EL}. It arises from an operation of braids : deleting one or several strings. An element of $IB_n$ is called a partial braid, and a product of two partial braids is defined (Section $1$ of \cite{EL}). Thus $IB_n$ is the monoid with product of partial braids.
\begin{proposition}[\cite{V2} Theorem 2.2.] \label{inv}
Let $IB_n$ be inverse braid monoid. Then 
\begin{equation*}
\mathcal{IM}(D^2\backslash\{1,2,\dots,n\})\simeq IB_n.
\end{equation*}
\end{proposition}
\section{$\mathbb{P}M$-mapping class monoids}
\subsection{Definitions}
We now define the $\mathbb{P}M$-mapping class monoid. Let $S=S_{g,b,n}$ as above. Let $m\ge 1$ and $\bm{f}=(f_1,\dots,f_m)$ be a sequence of orientation-preserving homeomorphism of $S$ which satisfy the following. There exists integers $1\le k_1<\dots<k_m<n$ and $\sigma,\tau\in S_n$ such that 
\begin{equation*}
f_l(\{\sigma(k_{l-1}+1),\dots,\sigma(k_l)\})=(\{\tau(k_{l-1}+1),\dots,\tau(k_l)\}),
\end{equation*}
and $f_l|_{\partial S}=id$ for all $l=1,\dots,m$ $(k_0=0,k_m=n)$. We consider the isotopy of $\bm{f}=(f_1,\dots,f_m)$ as simultaneously for all $l=1,\dots,m$. Denote the set of isotopy classes of such map by $\mathcal{PM}(S)$ and is called $\mathbb{P}M$-mapping class monoid. Let $\bm{h}=(h_1,\dots,h_{m^{'}})\in \mathcal{PM}(s)$ such that 
\begin{equation*}
h_l(\{\sigma^{'}(s_{l-1}+1,\dots,\sigma^{'}(k_l)\})=\{\tau^{'}(s_{l-1}+1),\dots,\tau^{'}(s_l)\}
\end{equation*}
for all $l=1,\dots,m^{'}$. Then product is defined by
\begin{equation}
\begin{split}
&(f_1,\dots,f_m)\cdot(h_1,\dots,h_{m^{'}}) \\
&=(f_1\circ h_1,f_2\circ h_1,\dots,f_m\circ h_1,\dots,f_1\circ h_{m^{'}},\dots,f_m\circ h_{m^{'}}), \label{seq}
\end{split}
\end{equation}
where 
\begin{equation*}
\begin{split}
&f_l\circ h_{l^{'}}(\{\sigma(k_{l-1}+1),\dots,\sigma(k_l)\}\cap\{\sigma^{'}(s_{l^{'}-1}+1),\dots,\sigma^{'}(s_l^{'})\}) \\
&=\{\tau(k_{l-1}+1),\dots,\tau(k_l)\}\cap\{\tau^{'}(s_{l^{'}-1}+1),\dots,\tau^{'}(s_l^{'})\},
\end{split}
\end{equation*}
and if 
\begin{equation*}
\{\sigma(k_{l-1}+1),\dots,\sigma(k_l)\}\cap\{\sigma^{'}(s_{l^{'}-1}+1),\dots,\sigma^{'}(s_l^{'})\}=\emptyset
\end{equation*}
then $f_l\circ h_{l^{'}}$ is removed in the right hand side of (\ref{seq}).
\subsection{Examples}
\begin{proposition} \label{pmc}
Let $S^2$ be a sphere. Then
\begin{equation*}
\mathcal{PM}(S^2\backslash\{1,2,3\})\simeq \mathscr{R}_3.
\end{equation*}
\end{proposition}
\begin{proof}
There is a surjective map 
\begin{equation*}
\phi:\mathcal{PM}(S^2\backslash\{1,2,3\})\rightarrow \mathscr{R}_3
\end{equation*}
using cases $\mathcal{M}(S^2\backslash\{1,2,3\})$ and $\mathcal{IM}(S^2\backslash\{1,2,3\})$. For $\bm{f}$, $\phi(\bm{f})$ is defined by the moves of $\{1,2,3\}$ by $\bm{f}$. We next prove the injectivity of $\phi$. Suppose
\begin{equation*}
\bm{f}=(f_1,\dots,f_m), \bm{g}=(g_1,\dots,g_{m^{'}})\in\mathcal{PM}(S^2\backslash\{1,2,3\}),
\end{equation*}
and $\phi(\bm{f})=\phi(\bm{g})$. Then $m=m^{'}$. If $m=1$ then as elements of $\mathcal{M}(S^2\backslash\{1,2,3\})$, $f_1$ and $g_1$ are isotopic. If  $m>1$ then as elements of inverse mapping class group, $f_l$ and $g_l$ are isotopic for $l=1,\dots,m$. Let the isotopies of $f_l$ and $g_l$ as $F_t^l$ for $l=1,\dots,m$. Then 
\begin{equation*}
\bm{F}_t^l=(F_t^1,\dots,F_t^m)
\end{equation*}
is a isotopy of $\bm{f}$ and $\bm{g}$.   
\end{proof}
\begin{proposition}
Let $\mathscr{RB}_n$ be $\mathbb{P}M$-braid monoid of $n$ strings. Then 
\begin{equation*}
\mathcal{PM}(D^2\backslash\{1,2,\dots,n\})\simeq \mathscr{RB}_n.
\end{equation*}
\end{proposition}
\begin{proof}
Let 
\begin{equation*}
\psi:\mathcal{IM}(D^2\backslash\{1,2,\dots,n\})\rightarrow IB_n
\end{equation*}
be an isomorphism of Proposition \ref{inv}. We define a map $\varphi:\mathcal{PM}(D^2)\rightarrow\mathscr{RB}_n$ by
\begin{equation*}
\varphi((f_1,\dots,f_m))=
(\psi(f_1), 
\psi(f_2),
\dots, 
\psi(f_m)),
\end{equation*}
where $\psi(f_i)$ denotes the i-layer of the element of $\mathbb{P}M$-monoid.
$\varphi$ is surjective since $\psi$ is surjective. We can prove injectivity of $\varphi$ by the same method of proof of Proposition \ref{pmc}.
\end{proof}

As a summary we have the following diagram
\begin{equation*}
\xymatrix{
S_n \ar@{^{(}-_>}[d] & B_n \ar@{->>}[l] \ar@{^{(}-_>}[d] \ar@{=}[r]  & \mathcal{M}(D^2\backslash \{1,2,\dots,n\})\\
\mathscr{R}_n  & \mathscr{RB}_n \ar@{->>}[l] \ar@{=}[r]  & \mathcal{PM}(D^2\backslash \{1,2,\dots,n\}).
}
\end{equation*}

\section{Main results}

\subsection{Dehn-Nilsen-Baer theorem for mapping groups}
\subsubsection{The non punctured case}
Let $S_g:=S_{g,0,0}$. The extended mapping class group, denoted $\mathcal{M}^{\pm}(S_g)$, is the group of isotopy classes of all homeomorphisms of $S_g$, including the orientation-reversing ones. For a group $G$, let $\Out(G)$ denote the outer automorphism group of $G$.
\begin{theorem}[\cite{FM} Theorem 8.1]
Let $g\geq 1$. The groups $\mathcal{M}^{\pm}(S_g)$ and $\Out(\pi_1(S_g))$ are isomorphism.
\end{theorem}
\subsubsection{The punctured case}
Let $\Out^{\star}(\pi_1(S))$ be the subgroup of $\Out(\pi_1(S))$ consisting of elements that preserve the set of conjugacy classes of the simple closed curves surrounding individual punctures. Let $S_{g,n}:=S_{g,0,n}$.
\begin{theorem}[\cite{FM} Theorem 8.8]
The groups $\mathcal{M}^{\pm}(S_{g,n})$ and $\Out^{\star}(\pi_1(S_{g,n}))$ are isomorphism.
\end{theorem}
\subsection{Dehn-Nilsen-Baer theorem for inverse mapping monoids}
We will review the Dehn-Nilsen-Bare theorem for inverse mapping monoids proved by R. Karoui and V. Vershinin \cite{V2}. We define the $(g,b,n)$-surface group as a group with the presentation
\begin{equation*}
\pi_{g,b,n}=\langle a_1,c_1,\dots,a_g,c_g,v_1,\dots,v_b,u_1,\dots,u_g\mid \displaystyle \prod_{i=1}^{n}u_i \displaystyle \prod_{l=1}^{b}v_l \displaystyle \prod_{m=1}^{g}[a_m,c_m] \rangle.
\end{equation*}
It is the fundamental group of a surface $S_{g,b,n}$. 

Let $H$ be a quotient group of $\pi_{g,b,n}$, defined by the conditions
\begin{equation*}
u_i=1 \text{ for all }i\not\in\{i_1,\dots,i_k\},
\end{equation*}
and let $K$ be a quotient group of $\pi_{g,b,n}$, defined by the conditions 
\begin{equation*}
u_j=1 \text{ for all }j\not\in\{j_1,\dots,j_k\}.
\end{equation*}
Let $t\in R_n$. Let $I_k=\{i_1,\dots,i_k\}$ be the domain of $t$ and $J_k=\{j_1,\dots,j_k\}$ be the image of $t$. Let $w_i$ be a word on letters 
\begin{equation*}
a_1,c_1,\dots,a_g,c_g,v_1,\dots,v_b,u_{j_1},\dots,u_{j_k}.
\end{equation*}
Let $I\Aut\pi_{g,b,n}$ be the monoid consisting of isomorphisms
\begin{equation*}
f_t:H\rightarrow K,
\end{equation*}
such that
\begin{equation*}
\begin{cases}
f_t(v_m)=v_m \text{ for }m=1,\dots,b, \\
f_t(u_i)=w_i^{-1}u_{t(i)}w_i, \text{ if } i \text{ is among } i_1,\dots,i_k,
\end{cases}
\end{equation*}
for all subgroups $H$ and $K$ of the type defined above, for all $k=0,1,\dots,n$. The composition of $f_t$ and $g_s$, $t,s\in R_n$, is defined for $u_i$ belonging to the domain of $t\circ s$. We put $u_{j_m}=1$ in a word $w_i$ if $u_{j_m}$ does not belong to the domain of definition of $g_s$.

We consider the case of empty boundary. Let $\pi_{g,n}:=\pi_{g,0,n}$. We will define an equivalence relation in $I\Aut\pi_{g,n}$. Let $f_1,f_2\in I\Aut\pi_{g,n}$. We define $f_1$ and $f_2$ are equivalent if there exists an element $q\in H$ such that 
\begin{equation*}
f_1(q^{-1}xq)=f_2(x).
\end{equation*}
We denote the quotient monoid by $I\Out\pi_{g,n}$.

Let us take an arbitrary element $\eta$ of $\mathcal{IM}(S_{g,n})$. It is represented by a homeomorphism of a surface
\begin{equation*}
h:S_{g,n}\backslash\{i_1,\dots,i_k\}\rightarrow S_{g,n}\backslash\{j_1,\dots,j_k\}.
\end{equation*}
It defines the bijection $\hat{h}$ between the conjugacy classes of 
\begin{equation*}
\pi_1(S_{g,n}\backslash\{i_1,\dots,i_k\}),
\end{equation*}
and 
\begin{equation*}
\pi_1(S_{g,n}\backslash \{j_1,\dots,j_k\}). 
\end{equation*}
We define $\hat{h}$ as an image of $\eta$ in $I\Out\pi_{g,n}$. Thus we can define a homomorphism of monoids 
\begin{equation*}
\psi_{g,n}:\mathcal{IM}(S_{g,n})\rightarrow I\Out\pi_{g,n};\eta \mapsto\hat{h}.
\end{equation*} 
The the following theorem holds.
\begin{theorem}[\cite{V2} Theorem 3.2] \label{DNBI}
The homomorphism $\psi_{g,n}$ is an isomorphism of monoids.
\end{theorem}
\subsection{Dehn-Nilsen-Baer theorem for $\mathbb{P}M$-mapping monoids}
Now we state the main result. Let $\mathbb{A}=(A_1,\dots,A_m)\in\mathscr{R}_n$. Let 
\begin{equation*}
I=(\{i_1,\dots,i_{k_1}\},\{i_{{k_1}+1},\dots,i_{k_2}\},\dots,\{i_{k_{m}+1},\dots,i_n\})
\end{equation*}
be the domain of $\mathbb{A}$ and
\begin{equation*}
J=(\{\sigma(i_1),\dots,\sigma(i_{k_1})\},\{\sigma(i_{{k_1}+1}),\dots,\sigma(i_{k_2})\},\dots,\{\sigma(i_{k_{m}+1}),\dots,\sigma(i_n)\})
\end{equation*}
be the image of $\mathbb{A}$. We define $H_l$ as a quotient group of $\pi_{g,b,n}$, defined by the conditions
\begin{equation*}
u_i=1 \text{ for all }i\not\in\{i_{k_{l}+1},\dots,i_{k_{l+1}}\}, 
\end{equation*}
for all $l=0,\dots,m$, where $k_0=0$, $k_{m+1}=n$, and $K_l$ as a quotient of $\pi_{g,b,n}$, defined by the conditions
\begin{equation*}
u_j=1 \text{ for all }j\not\in\{\sigma(i_{k_{l}+1}),\dots,\sigma(i_{k_{l+1}})\}, 
\end{equation*}
for all $l=0,\dots,m$, where $k_0=0$, $k_{m+1}=n$. Let $w_i$ be a word on letters 
\begin{equation*}
a_1,c_1,\dots,a_g,c_g,v_1,\dots,v_b,u_{j_1},\dots,u_{j_k}.
\end{equation*}
Let $\mathcal{R}\Out\pi_{g,b,n}$ be monoids of isomorphisms 
\begin{equation*}
f_{\mathbb{A}}:H_1\times\dots\times H_m\rightarrow K_1\times\dots\times K_m
\end{equation*}
such that
\begin{equation*}
\begin{cases}
f_{\mathbb{A}}((v_{r_1},\dots,v_{r_m})=(v_{r_1},\dots,v_{r_m}) \text{ for }r_1,\dots,r_m=1,\dots,b, \\
f_{\mathbb{A}}((u_{p_1},\dots,u_{p_m}))=(w_{p_1}^{-1}u_{A_1(i)}w_{p_1},\dots,w_{p_m}^{-1}u_{A_m(i)}w_{p_m}),\\
\,\,\,\,\,\,\,\,\,\,\,\,\,\,\,\,\,\,\,\,\,\,\,\,\,\,\,\,\,\,\,\,\,\,\,\,\,\, \text{ if } p_s \text{ is among } i_{s-1}+1,\dots,i_s \text{ for } s=1,\dots,m .
\end{cases}
\end{equation*}

We consider the case of empty boundary. Let $\pi_{g,n}:=\pi_{g,0,n}$. We will define an equivalence relation in $\mathcal{R}\Aut\pi_{g,n}$.Let $f_1,f_2$ be isomorphisms
\begin{equation*}
\begin{split}
f_1:H_1\times\dots\times H_{m_1}\rightarrow K_1\times\dots\times K_{m_1}, \\
f_2:H_1\times\dots\times H_{m_2}\rightarrow K_1\times\dots\times K_{m_2}.
\end{split}
\end{equation*}
The isomorphisms $f_1$ and $f_2$ are equivalent if $m_1=m_2$ and there exists elements $q_1\in H_1,\dots,q_{m_1}\in H_{m_1}$ such that
\begin{equation*}
f_1((q_1^{-1}x_1q_1,\dots,q_{m_1}^{-1}x_{m_1}q_{m_1}))=f_2((x_1,\dots,x_{m_1})).
\end{equation*}
We denote the quotient monoid by $\mathcal{R}\Out\pi_{g,n}$.

Let us take an arbitrary element $\eta$ of $\mathcal{PM}(S_{g,n})$. It is represented by a homeomorphism of a surface
\begin{equation*}
h:\displaystyle \prod_{s=1}^m  S_{g,n}\backslash\{i_{s-1}+1,\dots,i_s\}\rightarrow\displaystyle \prod_{s=1}^m S_{g,n}\backslash\{\sigma(i_{s-1}+1),\dots,\sigma(i_{s})\}.
\end{equation*}
It defines the bijection $\hat{h}$ between the conjugacy classes of 
\begin{equation*}
\pi_1(\displaystyle \prod_{s=1}^m  S_{g,n}\backslash\{i_{s-1}+1,\dots,i_s\}),
\end{equation*}
and 
\begin{equation*}
\pi_1(\displaystyle \prod_{s=1}^m S_{g,n}\backslash\{\sigma(i_{s-1}+1),\dots,\sigma(i_{s})\}). 
\end{equation*}
We define $\hat{h}$ as an image of $\eta$ in $\mathcal{R}\Out\pi_{g,n}$. Thus we can define a homomorphism of monoids 
\begin{equation*}
\psi_{g,n}:\mathcal{PM}(S_{g,n})\rightarrow \mathcal{R}\Out\pi_{g,n};\eta \mapsto\hat{h}.
\end{equation*} 
The the following theorem holds.
\begin{theorem}
The homomorphism $\psi_{g,n}$ is an isomorphism of monoids.
\end{theorem}
\begin{proof}
By our construction, we can consider $\psi_{g,n}$ as a map
\begin{equation*}
\psi_{g,n}:\displaystyle \bigsqcup_{\begin{matrix}n=n_1+\dots+n_k \\ k\ge 1 \end{matrix}} \displaystyle \prod_{i=1}^k \mathcal{IM}(S_{g,n_i})\rightarrow \displaystyle \bigsqcup_{\begin{matrix}n=n_1+\dots+n_k \\ k\ge 1 \end{matrix}} \displaystyle \prod_{i=1}^k I\Out\pi_{g,n_i}.
\end{equation*}
This map is an isomorphism by Theorem \ref{DNBI}.
\end{proof}

\end{document}